\documentclass[reqno,final]{amsart}
\usepackage{natbib}  
\usepackage{fancyhdr} 
\usepackage{color} 
\usepackage{hyperref} 
\usepackage{graphicx} 

\usepackage{pstricks}
\usepackage{amssymb}

\definecolor{aleacolor}{rgb}{0.16,0.59,0.78}

\hypersetup{
breaklinks,
colorlinks=true,
linkcolor=aleacolor,
urlcolor=aleacolor,
citecolor=aleacolor}

\renewcommand{\cite}{\citet}

\theoremstyle{plain}
\newtheorem{theorem}{Theorem}[section]                                          
\newtheorem{proposition}[theorem]{Proposition}                          

\newtheorem{corollary}[theorem]{Corollary}

\theoremstyle{definition}

\theoremstyle{remark}
\newtheorem{remark}[theorem]{Remark}
\newtheorem{example}[theorem]{Example}

\makeatletter \@addtoreset{equation}{section} \makeatother

\newcommand{\rme}{{\rm e}}
\newcommand{\rmd}{{\rm d}}

\newcommand{\ep}{{\varepsilon}}
\newcommand{\ind}{\mathbf{1}}
\newcommand{\pd}{{\partial}}

\newcommand{\PP}{{\mathbb P}}
\newcommand{\QQ}{{\mathbb Q}}

\newcommand{\EE}{{\mathbb E}}
\newcommand{\const}{\mathrm{const}}
\newcommand{\RN}{{\dfrac{\rmd\QQ}{\rmd\PP}}}

\newcounter{lis}

\begin{document}
\title[On  jump-diffusion processes with regime switching]
{On  jump-diffusion processes with regime switching: martingale approach }

\author{Antonio Di Crescenzo}
\author{Nikita Ratanov}
\address{
Dipartimento di Matematica, Universit\`a di Salerno,\newline
Via Giovanni Paolo II, 132 \newline 84084 - Fisciano (SA),   
 Italia}
 \email{adicrescenzo@unisa.it}
 \urladdr{\url{http://www.unisa.it/docenti/antoniodicrescenzo/index}}
 
 \address{Facultad de Econom\'{\i}a, Universidad del Rosario,\newline
Calle 12c, No.4-69\newline
Bogot\'a,  Colombia}     
 \email{nratanov@urosario.edu.co}
\urladdr{
\url{http://www.urosario.edu.co/Profesores/Listado-de-profesores/R/Nikita-Ratanov}}

\subjclass[2000]
{60G44; 60J75; 60K99}
\keywords{ jump-telegraph process; jump-diffusion process; 
martingales;  relative entropy; financial modelling}

\begin{abstract}
 We study  jump-diffusion processes with parameters switching at random times.
Being motivated by possible applications, 
we  characterise equivalent martingale measures for these processes by means of the
 relative entropy.
The minimal entropy approach is also developed. 
It is shown that in contrast to the case of  L\'evy processes,
for this model an Esscher transformation does not produce the minimal relative entropy.
\end{abstract}

\maketitle

\section{Introduction}

We investigate some basic properties of the jump-diffusion processes 
\[X(t)=T^c(t)+N^h(t)+W^\sigma(t),\quad t\geq0,\]
with time-dependent deterministic 
driving parameters  
switching simultaneously
at random times.
Here
 $W^\sigma$ denotes the Wiener part, defined by the stochastic integral (w.r.t. Brownian motion $B$)
of the  process which is formed by switching at random times
of the deterministic diffusion coefficients 
 $\sigma_i=\sigma_i(t),$ $i\in D:=\{1,\ldots, d\},\;d\geq2;$ 
by $N^h$ is denoted the jump part, i.e. the stochastic integral 
(w.r.t.  process $N=N(t)$ counting number of the regime switchings)
applied to the 
switching functions $h_i=h_i(t),$  $\;i\in D,$ 
 and $T^c$ is  
path-by-path integral in time $t$ of switching velocity regimes $c_i=c_i(t),\;i\in D$
(see the detailed definitions in Section \ref{sec:2}).
In the case of $d=2$ and exponentially distributed inter-switching time intervals
such processes are called  telegraph-jump-diffusion  processes, \cite{BJPS},
or Markov modulated jump-diffusion. 

In this paper the random inter-switching time intervals are assumed to be independent and arbitrary distributed.
In general, such a  process is not Markovian, 
and it is not a L\'evy process as well. 
We  study these processes from the martingale point of view,
including  Girsanov's measure transform.

Similar models without a diffusion component 
were considered first in \cite{MR} and more recently 
have been analysed in detail by \cite{DiC13-1, DiC13-2}, \cite{MCAP,STAPRO13,STAPRO14}, 
see also the particular cases in 
\cite{DiC01,DiC10}. The model with missing jump component is presented in \cite{DiC14,DiC13}.
The processes with random driving parameters are studied in \cite{STAPRO13}. 
The recent paper  \cite{SAA} is related to the 
model of random switching intensities.

This setting is widely used for applications, see e.g. \cite{weiss}.
The martingale approach developed in this paper is motivated by
 financial modelling, 
see  \cite{Rung} for jump-diffusion model. 
See also \cite{R2007} 
for jump-telegraph model
(and a more detailed presentation  in \cite{KR}).

The Markov modulated jump-diffusion model
of asset pricing 
(with additive jumps superimposed on the diffusion) 
 has been studied before, see  \cite{XG, BJPS}.  
 This model for a single risky asset  possesses infinitely many martingale measures, and thus the market is   
 incomplete. The  model can be completed 
 by  adding a further asset; for the jump-diffusion model see    \cite{Rung}, and 
 for the telegraph-jump-diffusion model see \cite{BJPS}. 
 
In this paper we explore another approach. We describe the set of equivalent martingale measures 
and determine  the F\"ollmer-Schweizer minimal probability measure
(so called the minimal entropy martingale measure (MEMM)), \cite{FS}. 
In his seminal paper \cite{Fritelli} Fritelli  has showed the equivalence between 
maximisation of expected exponential utility and the minimisation of the relative entropy.
Then, by this approach the models based on  L\'evy processes have been studied in \cite{FM}.

Observe that for L\'evy processes and for regime switching diffusions 
the usual technique is based on
an Esscher transform, which produces the MEMM, 
see \cite{FM,EscheSchweizer} and  \cite{ElliottAnnFin,Elliott2007}.

In our model this method does not work. The Esscher transform under regime switching
does not affect the switching intensities. 
In the L\'evy model \cite{FM}  this contradicts the minimal entropy condition
(if the  process is not already a martingale), 
see \eqref{cond:ME}.
In the case of the regime switching model \cite{ElliottAnnFin}
some  additional entropy given by
the jumps embedded into  this model can be reduced by more flexible measure transformation,
see Section \ref{sec:const}.

Moreover, in contrast with L\'evy model and with the regime switching diffusions without jumps
our model has the following important feature:
the entropy minimum  as well as  calibration of MEMM
depend on the time horizon under consideration.

The paper is organised as follows. Section \ref{sec:2} contains the definition and the main properties 
of the regime switching jump-diffusion processes with arbitrary distributions of inter-switching time intervals.
We construct Girsanov's transformation  in Section \ref{sec:3}. Then, we define the entropy   and derive
the corresponding Volterra equations.
In Section \ref{sec:4} we describe the set of equivalent martingale measures for 
 the regime switching jump-diffusion processes. The 
minimal entropy equivalent martingale measures are studied
in Section \ref{sec:const}    for the case of constant parameters.

Some applications including financial modelling will be presented elsewhere later.

\section{Generalised telegraph-jump-diffusion processes. Distributions and expectations}\label{sec:2}
\setcounter{equation}{0}

Let  $(\Omega, \mathcal{F}, \PP)$ be a complete probability space
with the  given right-continuous filtration $\mathcal{F}_t,\; t\geq0,$ satisfying the usual hypotheses, \cite{JYC}.
We start with a $d$-state  $\mathcal{F}_t$-adapted 
semi-Markov (see \cite{Jacobsen})
random  process $\ep,\;$  $\ep(t)\in D,$ $\;t\geq0.$ The switchings occur at random times 
$\tau_n,\;n\geq0,\;\tau_0=0,$ and process $\ep$ is  right-continuous with left-hand limits.

   Let $N=N(t)$ be the counting process, $N(t)=\max\{n~|~\tau_n\leq t\},\;t\geq0$.
   
   \subsection{The definition of telegraph-jump-diffusion process}
For the set of  deterministic measurable functions
 $z_i=z_i(t),\;t\geq0, \;i\in D,$ 
we construct first the piecewise deterministic random process 
$z^\dagger$  combining $z_i,\;i\in D,$  by means of 
 the switching process $\ep$: 
  \begin{equation}
\label{def:transform}
z^\dagger(t)=\sum_{n=1}^\infty
z_{\ep(\tau_{n-1})}(t-\tau_{n-1})\ind_{\Delta_n}(t),\qquad t\geq0.
\end{equation}
Here $\Delta_n:=[\tau_{n-1},\;\tau_{n}),\;n\geq1,$ and $\ind_\Delta(\cdot)$ is the indicator function.
Process $z^\dagger$ starts from the origin at the switching time $\tau_0=0$;
at each further switching time $\tau_n,\;n\geq1,$ the process $z^\dagger$ 
is renewed.

Second, consider the integrals of $z^\dagger$ of the following three types:
 \begin{enumerate}
  \item  the generalised $d$-state telegraph process    (path-by-path integral)
\begin{equation}
\label{def:TP}
T^z(t)=\int_0^tz^\dagger(u)\rmd u,\qquad t>0;
\end{equation}
  \item the pure jump process (integral w.r.t. counting process $N$)
\begin{equation}
\label{def:JP}
N^z(t)=\int_0^tz^\dagger(u)\rmd N(u),\qquad t>0,
\end{equation}
  \item  the Wiener process (It\^o integral)
\begin{equation}
\label{def:DP}
W^z(t)=\int_0^tz^\dagger(u)\rmd B(u),\qquad t>0,
\end{equation}
where  $B=B(t),\;t\geq0,$ is an $\mathcal{F}_t$-adapted Brownian motion, 
independent of $\ep$.
Note that $W^z=W^z(t),\;t\geq0,$ is a Gaussian $(\mathcal F_t, \PP)$-martingale. 
\end{enumerate}

Let $c_i=c_i(t),\; h_i=h_i(t)$ and  $\sigma_i=\sigma_i(t),\;t\geq0,\; i\in D,$
 be   deterministic measurable functions. 
We assume that functions 
$c_i$ are locally integrable, and
 $\sigma_i$ are locally square integrable, 
\begin{equation*}
\int_0^t\sigma_i(u)^2\rmd u
<\infty,\qquad t>0,\;i\in D.
\end{equation*}

In this paper we analyse  the  jump-diffusion process with switching regimes
$X=X(t),\;t\geq0,$ of the following form
\begin{equation}
\label{def:JTD}
X(t):=T^c(t)+N^h(t)+W^\sigma(t), \qquad t\geq0,
\end{equation}
with the components $T^c,\; N^h$ and $W^\sigma$ which  are defined 
by \eqref{def:TP}, \eqref{def:JP} and \eqref{def:DP}  respectively. Process $X$ satisfies the stochastic equation
\begin{equation*}
\label{def:JDTdiff}
\rmd X(t)=c_{\ep(\tau_{N(t)})}(t-\tau_{N(t)})\rmd t
+h_{\ep(\tau_{N(t)})}(T_{N(t)})\rmd N(t)
+\sigma_{\ep(\tau_{N(t)})}(t-\tau_{N(t)})\rmd B(t),\quad t>0.
\end{equation*}

Equivalently, processes $T^c,\; N^h,\; W^\sigma$ can be expressed by summing up  the 
 paths between the consequent switching instants $\tau_n$:
\begin{align}
\label{def:Tsum}
T^c(t)=&\int_0^tc^\dagger(u)\rmd u=\sum_{n=1}^{N(t)}l_{\ep(\tau_{n-1})}(T_n)+l_{\ep(\tau_{N(t)})}(t-\tau_{N(t)}),\\
\label{def:Jsum}
N^h(t)=&\int_0^th^\dagger(u)\rmd N(u)=\sum_{n=1}^{N(t)}h_{\ep(\tau_{n-1})}(T_n),\\
\label{def:Dsum}
W^\sigma(t)=&\int_0^t\sigma^\dagger(u)\rmd B(u)
=\sum_{n=1}^{N(t)}w_{\ep(\tau_{n-1})}(T_n)+w_{\ep(\tau_{N(t)})}(t-\tau_{N(t)}),
\end{align}
where  the following notations are used
\begin{equation}
\label{def:liwi}
l_i(t)=\int_0^tc_i(u)\rmd u,\qquad w_i(t)=\int_0^t\sigma_i(u)\rmd B(u),\quad t\geq0,\;i\in D.
\end{equation}
Note that $l_i(t),\;i\in D,$ are deterministic 
and the variables $w_i(t),\;i\in D,$ are   zero-mean normally distributed with 
the variances 
$\Sigma_i(t)^2:=\int_0^t\sigma_i(u)^2\rmd u,\;t>0,\; i\in D$.

We say that the regime switching  jump-diffusion process $X$   defined by \eqref{def:JTD}-\eqref{def:liwi}
 is characterised by the triplet
$\langle c, h, \sigma\rangle$ with distributions of inter-switching time intervals $T_n$
which are determined by the     hazard rate functions $\gamma_{ij}=\gamma_{ij}(t),$ $i, j\in D,$ 
 (see  the definition in    \eqref{def:hazard}).

We apply the notations $X_i$ and $T_i^c, N_i^h, W_i^\sigma$, if 
 the initial state $i\in D$ of the underlying process $\ep$ is given, $\ep(0)=i$.

Further, we will need the following explicit expression for the stochastic
 exponential $\mathcal E_t (X)$ of $X=T^c+N^h+W^\sigma$.
It is known, see e.g. \textup{\cite{JYC},}  that
\begin{equation}
\label{eq:stochexp}
\mathcal E_t (X)
=\exp\left(T^{c-\sigma^2/2}(t)+W^\sigma(t)\right)
\prod_{n=1}^{N_t}(1+h_{\ep(\tau_{n})}(T_n))
=\exp\left(Y(t)\right),
\end{equation}
where $Y(t)=T^{c-\sigma^2/2}(t)+N^{\ln(1+h)}(t)+W^\sigma(t),\;t\geq0.$
Here 
$
T^{c-\sigma^2/2}
$
is the telegraph process 
\eqref{def:Tsum} with the    velocities $c_i-\sigma_i^2/2$ instead of $c_i,\; i\in D,$ 
and  $N^{\ln(1+h)}$ is the pure jump process \eqref{def:Jsum}
with  the  jump values $\ln(1+h_i),$ instead of $h_i,\;i\in D,$ 
switching at random times $\tau_n,\;n\geq1.$

\subsection{Semi-Markov process $\ep=\ep(t)$}
 
In order to state the distribution of $X(t)$ 
we  introduce conditions on the driving processes $\ep$ and $\{\tau_n\}_{n\geq1}$.

Denote by $\overline F_{ij}=\overline F_{ij}(t),\;t>0,\;i, j\in D,$
the transition probabilities of the form \cite[(3.17)]{Jacobsen}:

\[\overline F_{ij}(t)=\PP\{\tau_1>t,\; \ep(\tau_1)=j~|~\ep(0)=i\},\quad t>0,\;i, j\in D.\]

Let $\overline F_{ij}(t)>0,\;t>0,\;i, j\in D.$
Consider the hazard functions $\Gamma_{ij},$  
\[
\Gamma_{ij}(t)=-\ln\overline F_{ij}(t),\quad t>0,\;i, j\in D.
\]
Let functions $\overline F_{ij}$ be differentiable, 
$\dfrac{\rmd \overline F_{ij}}{\rmd t}(t)=-f_{ij}(t),\;t>0,\;i, j\in D$.
Thus the hazard functions $\Gamma_{ij}=\Gamma_{ij}(t)$ are expressed by
\[
\Gamma_{ij}(t)=\int_0^t\gamma_{ij}(u)\rmd u,\qquad t\geq0.
\]
Here 
\begin{equation}
\label{def:hazard}
\gamma_{ij}(u):=\frac{f_{ij}(u)}{\overline F_{ij}(u)},\quad \;u>0,\qquad i, j\in D,
\end{equation}
are the hazard rate functions, and
 $f_{ij},\; i,j\in D,$ are the  density functions of the interarrival times.
    We assume the non-exploding condition to be hold:
\begin{equation}
\label{def:blowup}
\int_0^\infty\gamma_{ij}(u)\rmd u=+\infty,\qquad i, j\in D.
\end{equation}

Note that 
\begin{equation}\label{barFf}
\overline F_{ij}(t)=\exp\left(
-\int_0^t\gamma_{ij}(u)\rmd u
\right),\quad
f_{ij}(t) =\gamma_{ij}(t)\exp\left(
-\int_0^t\gamma_{ij}(u)\rmd u
\right),\quad t>0,\quad i, j\in D.
\end{equation}

 The survival function of the first switching time $\tau_1$ is given by 
\begin{equation}\label{eq:barFibarFij}
\overline F_{i}(t):=\PP\{\tau_1>t~|~\ep(0)=i\}=\prod_{j\in D\setminus\{i\}}\overline F_{ij}(t)
=\exp\left(-\int_0^t\gamma_i(u)\rmd u\right), \quad t>0,\qquad i\in D,
\end{equation}
where 
\begin{equation}
\label{def:hazardS}
\gamma_i=\sum\limits_{j\in D\setminus\{i\}}\gamma_{ij},\quad i\in D.
\end{equation}
Furthermore, let 
\begin{equation}
\label{eq:fi}
f_i(t)=-\frac{\rmd \overline F_i}{\rmd t}(t)
=\gamma_{i}(t)\exp\left(
-\int_0^t\gamma_{i}(u)\rmd u
\right), 
\quad t>0,\qquad i\in D.
\end{equation}

    Due to \eqref{barFf} $N$ is an inhomogeneous Poisson process with switchings at $\tau_n,\; n\geq1,$
 and with the instantaneous   intensities $\gamma_{ij}(t),\; t>0,\; i, j\in D$.

If the process $\ep$ is observed beginning from the time  $s,\; \tau_0\leq s<\tau_1,$
the corresponding conditional distributions can be described by the following conditional  survival functions,
\begin{equation}\label{barFbarf}
\begin{aligned}
\overline F_{ij}(t~|~s)=&\PP(\tau_1>t,\;\ep(\tau_1)=j~|~\tau_1>s,\; \ep(0)=i)
=\frac{\overline F_{ij}(t)}{\overline F_{i}(s)},
\end{aligned}\end{equation}
with the density functions 
\[
f_{ij}(t~|~s)=-\frac{\pd}{\pd t}\overline F_{ij}(t~|~s)=\frac{f_{ij}(t)}{\overline F_{i}(s)},\quad 0\leq s<t,
\quad i, j\in D.
\]
Moreover, let
\begin{equation}
\label{eq:barFibarFijs}\begin{aligned}
\overline F_i(t~|~s):=
\PP(\tau_1>t~|~\tau_1>s,\;\ep(0)=i)=&\exp\left(-\int_s^t\gamma_i(u)\rmd u\right)\\
f_i(t~|~s)=\gamma_i(t)&\exp\left(-\int_s^t\gamma_i(u)\rmd u\right),
\end{aligned}\qquad i\in D,\end{equation}
see \eqref{eq:barFibarFij}.

Notice that
\[\begin{aligned}
\overline F_{ij}(t~|~0)\equiv& \overline F_{ij}(t),\qquad \overline F_i(t~|~0)\equiv\overline F_i(t),\\
f_{ij}(t~|~0)\equiv& f_{ij}(t),\qquad f_i(t~|~0)\equiv f_i(t),
\end{aligned}
\qquad t>0,\qquad i,j\in D.\]

Assume the inter-switching  time intervals $T_n=\Delta \tau_n=\tau_{n}-\tau_{n-1},\;n\geq1,$
 to be independent, and process $\ep$ to be renewal in the following sense: $n\geq1,$
\begin{equation*}
\PP\{T_n>t~|~T_n>s,\; \ep(\tau_{n-1})=i\}=\overline F_{i}(t~|~s), 
\quad  t>s\geq0,\qquad   i\in D,%
\end{equation*}
see \eqref{barFbarf}-\eqref{eq:barFibarFijs}.

\subsection{The distribution and expectation of $X(t)$}

Our further analysis is based on the following observation. 

Let $\tau_1=\tau_1^{(i)}$ be the first switching time,
where $i\in D$ is the fixed 
 initial state  of the process $\ep,\; \ep(0)=i.$
Owing to 
the renewal character of the  counting process
$N$, see \cite{renewal},
we have the following  equalities in distribution:  
\begin{equation}
\label{eq:distrX}
\begin{aligned}
X_i(t)\stackrel{D}{=}(l_i(t)+w_i(t))\ind_{\tau_1>t}+&\left[l_i(\tau_1)+w_i(\tau_1)+h_i(\tau_1)
+\tilde X_{\ep_i(\tau_1)}(t-\tau_1)\right]\ind_{\tau_1<t},\\
& t>0,\;i\in D,
\end{aligned}\end{equation}
where $\tilde X$ is the regime switching  jump-diffusion process independent of $X$
which starts at time $\tau_1$.
 Here $\ind_A$ is the indicator 
of event $A$.

Let
\begin{equation*}
\label{def:dens}
\begin{aligned}
p_i(x, t):=\PP(X_i(t)\in\rmd x)/\rmd x,\quad &t\geq0,
\\
p_i(x, t~|~s):=\PP(X_i(t)\in\rmd x~|~N_i(s)=0)/\rmd x,\quad &t\geq s,
\end{aligned}
\qquad x\in(-\infty, \infty),\;i\in D,
\end{equation*}
be the  density functions of $X_i(t),\;i\in D$. 
Note that $p_i(x, t~|~0)=p_i(x, t)$,  $i\in D$.

In these terms equalities  \eqref{eq:distrX} take
the following form 
\begin{equation}\begin{aligned}
\label{eq:p0p1}
p_i(x, t~|~s)=&\psi_i(x-l_i(t), t)\overline F_i(t~|~s)   \\
 +&\sum_{j\in D\setminus\{i\}}\int_s^t\left(
    \int_{-\infty}^\infty p_j(x-l_i(u)-h_i(u)-y,   t-u)\psi_i(y, u)\rmd y
    \right)f_{ij}(u~|~s)\rmd u,\\
  &  t>s\geq0,\; x\in(-\infty, \infty),\;  i\in D.
\end{aligned}\end{equation}
Here $\psi_i=\psi_i(x, t)$ is the density function of the Gaussian random variable $w_i(t),\; t>0,\;i\in D,$
\[
\psi_i(x, t)=\frac{1}{\Sigma_i(t)\sqrt{2\pi}}\exp\left(
-\frac{x^2}{2\Sigma_i(t)^2}
\right),\qquad x\in(-\infty, \infty),
\]
where $\Sigma_i(t)^2:=\int_0^t\sigma_i(u)^2\rmd u,\;t>0,\; i\in D$.

In the Markov case of two-state processes, $d=2,$ with  constant parameters $c_i, h_i, \sigma_i,$ 
$i\in\{1, 2\},$
and with   exponentially distributed inter-switching times $T_n,$
the distributions of $X_1(t)$ and $X_2(t)$ have been analysed in detail in \cite{BJPS}.

Let us study the expectations 
\[
\mu_i(t):=\EE[X_i(t)],\qquad t\geq0,\;i\in D,
\]
 and 
\[
\mu_i(t~|~s):=\EE[X_i(t)~|~ N_i(s)=0],\qquad
 t\geq0,\;i\in D,
\]
of the $d$-state process $X=X(t),\;t>0.$

Note that $\EE[W^\sigma(t)]=0,\;\forall t\geq0$.
Moreover, for $t\in[0, s],$
$\mu_i(t~|~s)=l_i(t),\;i\in D$.

To characterise $\mu_i(t),\;\mu_i(t~|~s),\;i\in D,$
we use the following notations:
\[\begin{aligned}
a_i(t):=\int_0^t\left[
c_i(u)\overline F_{i}(u)+h_i(u)f_{i}(u)
\right]\rmd u,\qquad &t\geq0,
\\
a_i(t~|~s):=l_i(s)+\int_s^t[c_i(u)\overline F_{i}(u~|~s)+h_i(u)f_{i}(u~|~s)]\rmd u,\qquad &t\geq s\geq0,
\end{aligned}
\quad i\in D.
\]
Here  $f_i(u),\;  \overline F_i(u)$ and 
$f_i(u~|~s),\; \overline F_i(u~|~s),\; i\in D,$ are defined in \eqref{eq:barFibarFij}, \eqref{eq:fi} and \eqref{eq:barFibarFijs}.

\begin{proposition}\label{prop:mu}
The expectations $\mu_i(t),\;t\geq0,\;i\in D,$ 
satisfy the following Volterra system of integral equations\textup{,} 
\begin{equation}\label{eq:mu0mu1}
    \mu_i(t)=a_i(t)+\sum_{j\in D\setminus\{i\}}\int_0^t\mu_j(t-u)f_{ij}(u)\rmd u,
\qquad t\geq0,\;i\in D.
\end{equation}

If functions $\mu_i(t),\;t\geq0,\;i\in D,$ solve system \eqref{eq:mu0mu1}, 
then the conditional expectations $\mu_i(t~|~s)$ are given by
\begin{equation}\label{eq:mu0mu1s}
    \mu_i(t~|~s)=a_i(t~|~s)+\sum_{j\in D\setminus\{i\}}\int_s^t\mu_j(t-u)f_{ij}(u~|~s)\rmd u,
\qquad t\geq s,\;i\in D.
\end{equation}
\end{proposition}

\begin{proof} 
By applying \eqref{eq:p0p1} (with $s=0$)
to
\[
\mu_i(t)=\int_{-\infty}^\infty xp_i(x, t)\rmd x
\]
one can get 
\[\begin{aligned}
\mu_i(t)=&l_i(t)\overline F_i(t)+\sum_{j\in D\setminus\{i\}}
\int_0^t\left(
\int_{-\infty}^\infty\psi_i(y, u)\rmd y\int_{-\infty}^\infty xp_j(x-l_i(u)-h_i(u)-y, t-u)\rmd x
\right)f_{ij}(u)\rmd u\\
=&l_i(t)\overline F_i(t)+\sum_{j\in D\setminus\{i\}}\int_0^t
\left[
\mu_j(t-u)+l_i(u)+h_i(u)
\right]f_{ij}(u)\rmd u.
\end{aligned}\]
Integrating by parts we have 
\[
\mu_i(t)=l_i(t)\overline F_i(t)+\sum_{j\in D\setminus\{i\}}\left[
\int_0^t\left(\mu_j(t-u)+h_i(u)\right)f_{ij}(u)\rmd u
-l_i(t)\overline F_{ij}(t)+\int_0^tc_i(u)\overline F_{ij}(u)\rmd u
\right],
\]
which gives \eqref{eq:mu0mu1}. The proof of \eqref{eq:mu0mu1s} is similar.\end{proof}

In the case $d=2$
equations \eqref{eq:mu0mu1} are derived e.g. in equations (3.2)-(3.3) of \cite{STAPRO13}.

\begin{corollary}\label{corollary}
The identities 
\[
\mu_i(t)\equiv0,\;t\geq0,\text{~~~ and~~~}
\mu_i(t~|~s)\equiv l_i(s),\; t\geq s,\qquad i\in D,
\]
hold 
if  and only if
\begin{equation}\label{eq:DoobMeyer}
\gamma_i(t)h_i(t)+ {c_i(t)} \equiv0,
\qquad t\geq0,\;i\in D,
\end{equation} 
where $\gamma_i=\gamma_i(t),\; t\geq0,\;i\in D,$ are the   hazard rate functions\textup{,}
which  are defined by \eqref{def:hazard}, \eqref{def:hazardS}.
\end{corollary}

\begin{proof}
Notice that systems \eqref{eq:mu0mu1} and \eqref{eq:mu0mu1s} have  the trivial solutions
$\mu_i(t)\equiv0,$ $t\geq0,\;i\in D$ and
$\mu_i(t~|~s)\equiv l_i(s),\;t\geq s,\;i\in D,$ respectively, if and only if
$a_i(t)\equiv0$ and $a_i(t~|~s)\equiv l_0(s),\; i\in D$. 
These equalities hold, when
\begin{equation*}
\label{eq:DM}\begin{aligned}
c_i(u)\overline F_i(u)+h_i(u)f_i(u)\equiv0,\\
c_i(u)\overline F_i(u~|~s)+h_i(u)f_i(u~|~s)\equiv0,
\end{aligned}\qquad u>s>0,\;  i\in D,
\end{equation*}
which is equivalent to \eqref{eq:DoobMeyer}, due to  \eqref{def:hazard}-\eqref{barFf} and \eqref{barFbarf}.
\end{proof}
 
\begin{remark}
Condition \eqref{eq:DoobMeyer} has the sense of Doob-Meyer decomposition.
These type of conditions for the jump-telegraph processes appears first in 
\textup{\cite[Theorem 1]{R2007} }
in the case of constant deterministic parameters $c, h, \gamma$
\textup{(} see also \textup{\cite{KR}).} In this case condition \eqref{eq:DoobMeyer}
is intuitively evident. It means that the displacement performed by the telegraph process
during a time-period $\tau$ equal to the mean-switching-time
 is identical to the jump's size performed in the opposite direction. 

This intuitively explains why this is a martingale condition. 
\end{remark}

\section{Girsanov's transformation}\label{sec:3}
\setcounter{equation}{0}
In this section we analyse the problems which are important for applications, e.g. for financial modelling.  
First, we describe all possible martingales in our setting. Then, we derive a generalisation of
Girsanov's Theorem. 

\subsection{Martingale's characterisation}

Since the diffusion part $W^\sigma=W^\sigma(t)=\int_0^t\sigma^\dagger(u)\rmd B(u)$ 
 is already a $\PP$-martingale, 
 it is sufficient to investigate the process
 $X=T^c(t)+N^h(t),\;t\geq0.$

\begin{theorem}\label{theo:DoobMeyer}
Let $X=X(t),\;t\geq0,$ 
be the  process 
with  the parameters
$\langle c_i, h_i\rangle,\;i\in D,$  switching at random times $\tau_n,\;n\geq0.$ 
Let the inter-switching times 
$T_n=\tau_n-\tau_{n-1},\;n\geq1,$ be distributed with    
hazard rate functions $\gamma_i=\gamma_i(t),\;t\geq0,\; i\in D,$ see \eqref{def:hazard}, \eqref{def:hazardS}.

Process $X$ is a martingale if and only if the equalities in  \eqref{eq:DoobMeyer} are fulfilled\textup{.}
\end{theorem}

\begin{proof}
If $X$ is a martingale, then $\mu_0(t)=\mu_1(t)\equiv0$, which is equivalent to \eqref{eq:DoobMeyer},
see Proposition \ref{prop:mu}. 

Conversely, it is known, see \cite{MM}, Proposition 2.13, 
that the compensated jump process
$N^h-T^{\gamma h}$ is a martingale. Therefore,
if identities \eqref{eq:DoobMeyer} hold, then $-T^c$ is the compensator of $N^h$
and the sum $T^c+N^h$ is a martingale. 
\end{proof}

Notice that if jumps  vanish\textup{,} $h_i\equiv0,\;i\in D,$ 
 process $X$ is a martingale 
only in the trivial case\textup{:}  
$c_i\equiv0,\;i\in D,$ and thus $X=0,$ see  \eqref{eq:DoobMeyer}.

\begin{corollary}[\cite{BJPS}]\label{cor:const}
Let $X=X(t),\;t\geq0,$ be the jump-diffusion  process with switching \emph{constant} parameters 
$c_i, h_i$ and $\sigma_i,\;i\in D$. Let $h_i\neq0,\;i\in D$.

Process $X$ is a martingale if and only if  
$
c_i/h_i<0,\; i\in D,
$
and the 
distributions of the inter-switching times $T_n$ are 
\emph{exponential,} $\mathrm{Exp}(\lambda_i),$ with  parameters 
$\lambda_i,\; \lambda_i=-c_i/h_i>0,\;$
$i\in D$. In this case the underlying   $\ep$ is a Markov process.
\end{corollary}

\begin{remark}
In the paper  by  Di Crescenzo et al  \textup{\cite{DiC14}}
the generalised $2$-state geometric  telegraph-diffusion process $S=S(t)$ 
with constant parameters $c_0,$ $c_1$ and $\sigma$ is studied\textup{,}
\[
S(t)=s_0\exp\left[
T(t)+\sigma B(t)
\right], 
\] 
where $B$ is a standard Brownian motion and the inter-switching times are independent and arbitrarily distributed.
The authors expected that  the process
$S=S(t)$ can be transformed in a martingale by superimposing of a jump component. 
This expectation is not justified.

The process $S(t)/s_0$  is the stochastic exponential  
of $X=T(t)+\sigma B(t)+\sigma^2t/2.$ 
After the inclusion of a jump component with \emph{constant}
jump amplitudes $h_1, h_2>-1$\textup{,} such that $\dfrac{c_i+\sigma^2/2}{h_i}<0,$ $i\in D=\{1, 2\},$
 processes $X$ and $S$ become    martingales only in the standard case of exponentially 
distributed inter-arrival times, $\mathrm{Exp}(\lambda_i),$ with constant intensities  
\[
\lambda_i=-\frac{c_i+\sigma^2/2}{h_i},\quad i\in D=\{1, 2\},
\]
\textup{(}see Corollary \textup{\ref{cor:const}).}
\end{remark}


\subsection{Girsanov's Theorem}
The problem of existence and uniqueness  of an equivalent martingale measure 
is extremely significant for applications, especially in the  theory of financial derivatives.
It is important to 
understand how the equivalent martingale measures can be constructed
 if such a measure exists.

Let $\ep(t)\in D,\;t\geq0,$ be the switching process 
on the filtered probability space $(\Omega,   \mathcal F, \mathcal F_t, \PP)$
governed by  the   hazard rate functions 
$\gamma_i=\gamma_i^\PP(t)=f_i(t)/\overline F_i(t),\;t>0, i\in D,$ of the  inter-switching times, 
see  \eqref{def:hazardS}, \eqref{def:hazard},
satisfying the non-exploding condition \eqref{def:blowup}.

Let $c_i^*$ and  $h_i^*,\;i\in D,$ be measurable functions 
satisfying the martingale condition \eqref{eq:DoobMeyer},  
\begin{equation}
\label{eq:DoobMeyer*}
\gamma^\PP_i(t)h_i^*(t)+c_i^*(t)\equiv0\qquad  t\geq0,\; i\in D.
\end{equation} 
We assume 
$c_i^*,\;i\in D,$ to be  locally integrable and $h_i^*(t)>-1,\;t\geq0,\;i\in D.$  Thus,
\begin{equation}
\label{eq:c*gamma}
c_i^*(t)\leq\gamma_i^\PP(t),\qquad \forall t>0,\;i\in D.
\end{equation}
Furthermore, let
\begin{equation}
\label{eq:c*g}
\int_0^\infty\left(
c_i^*(u)-\gamma_i^\PP(u)
\right)\rmd u=-\infty,\qquad i\in D.
\end{equation}

Consider the  jump-diffusion martingale $X^*=T^{*}(t)+N^{*}(t)+W^*(t)$ 
with regime switching,
where
\begin{equation}
\label{def:JTDsum*}
\begin{aligned}
T^*(t)=&\sum_{n=1}^{N(t)}l_{\ep(\tau_{n-1})}^*(T_n)+l_{\ep(\tau_{N(t)})}^*(t-\tau_{N(t)}),\\
N^*(t)=&\sum_{n=1}^{N(t)}h_{\ep(\tau_{n-1})}^*(T_n),\\
W^*(t)=&\sum_{n=1}^{N(t)}w_{\ep(\tau_{n-1})}^*(T_n)+w_{\ep(\tau_{N(t)})}^*(t-\tau_{N(t)}).
\end{aligned}\end{equation}
Here, see \eqref{def:liwi},
\[
l_i^*(t)=\int_0^tc_i^*(u)\rmd u,\qquad w_i^*(t)=\int_0^t\sigma_i^*(u)\rmd B(u), 
\]
where $\sigma_i^*,\; i\in D,$ are locally square integrable functions.

Let $Z=Z(t)=\mathcal E_t(X^*)$ be the stochastic exponential of $X^*$. By  \eqref{eq:stochexp}
\begin{equation}
\label{def:RN1}
\begin{aligned}
Z(t)=\mathcal E_t(X^*)
=&\exp\left(T^{c^*-{\sigma^*}^2/2}(t)+W^{\sigma^*}(t)\right)
\prod_{n=1}^{N_t}(1+h^*_{\ep(\tau_{n-1})}(T_n))\\
=&\exp\left(T^{c^*-{\sigma^*}^2/2}(t)+N^{\ln(1+h^*)}+{W^\sigma}^*(t)\right).
\end{aligned}
\end{equation}
Here 
$
T^{c^*-{\sigma^*}^2/2}
$
is the $d$-state generalised telegraph process 
\eqref{def:Tsum} with the velocity regimes $c_i^*-(\sigma_i^*)^2/2$ instead of $c_i,$  and 
$N^{\ln(1+h^*)}$ is the pure jump process \eqref{def:Jsum}
with the  jump values $\ln(1+h_i^*)$ instead of $h_i,$ $i\in D.$

\begin{theorem}[Girsanov's Theorem]\label{theo:Girsanov}
Assume that conditions \eqref{eq:DoobMeyer*}-\eqref{eq:c*g} hold.

Let measure $\QQ$ be equivalent 
 to $\PP$ under the fixed time horizon $t,\; t\geq0,$
 with the density 
\begin{equation}
\label{def:RN2}
\frac{\rmd \QQ}{\rmd\PP}|_{\mathcal F_t}=Z(t). 
\end{equation}

Under the measure $\QQ:$
\begin{list}{\textup{(\alph{lis})}}
{\usecounter{lis}}
\item\label{first}
 the inter-arrival times $\{T_n=\tau_n-\tau_{n-1}\}_{n\geq1}$ 
are independent and distributed with the   survival functions 
\begin{equation}\label{eq:SF*}
\overline F_i^\QQ(t)=\exp(l_i^*(t))\overline F_i^\PP(t),\qquad t\geq0,\; i\in D.
\end{equation}
The hazard rate functions $\gamma_i^\QQ$ of these distributions are given by
\begin{equation}
\label{eq:HRF*}
\gamma_i^\QQ(t)=\gamma_i^\PP(t)-c_i^*(t)
\equiv(1+h_i^*(t))\gamma_i^\PP(t),\qquad t\geq0,\; i\in D,
\end{equation}
and  the non-exploding condition
\begin{equation}
\label{eq:blowupQ}
\int_0^\infty\gamma_i^\QQ(u)\rmd u=+\infty,\qquad i\in D,
\end{equation}
holds\textup{;}

\item\label{second}
the process $\widetilde B(t)=B(t)-L^*(t)$ is the standard $\QQ$-Brownian motion\textup{,}
where $L^*(t),$  $t\geq0,$ is the generalised  telegraph process with switching velocities 
$\sigma_i^*,\;i\in D,$ i.e.
\[
L^*(t):=T^{\sigma^*}(t)=\sum_{n=1}^{N(t)}\int_0^{T_n}\sigma_{\ep(\tau_{n-1})}^*(u)\rmd u
+\int_0^{t-\tau_{N(t)}}\sigma_{\ep(\tau_{N(t)})}^*(u)\rmd u;
\]

\item\label{third}
 the  $\PP$-jump-diffusion process $X$ with switching regimes\textup{,}
 which is defined by \eqref{def:JTD}\textup{,}
 \[X=T^c(t)+N^h(t)+W^\sigma(t) ,\;t\geq0,\] 
under measure $\QQ$  is still a jump-diffusion process  
  \[X=T^{c+\sigma\sigma^*}(t)+N^h(t)+\widetilde W^\sigma(t),\;t\geq0,\]
characterised by $\langle c+\sigma\sigma^*, h, \sigma\rangle,$ 
with the hazard rate functions $\gamma_i^\QQ$ 
of the inter-switching times $\{T_n\}_{n\geq1},$ determined by \eqref{eq:HRF*}.
Here $\widetilde W^\sigma$ is the stochastic integral
\eqref{def:Dsum} 
based on the $\QQ$-Brownian motion  
$\widetilde B.$ 
\end{list}
\end{theorem}

\begin{proof}
By definition, see \eqref{def:RN1}-\eqref{def:RN2}, we have 
\begin{equation*}
\label{eq:QbarF}
\begin{aligned}
\overline F_i^\QQ(t)=&\QQ\{\tau_1>t~|~\ep(0)=i\}
=\EE_\PP\left\{Z(t)\ind_{\{\tau_1>t\}}~|~\ep(0)=i\right\}\\
=&\exp\left(\int_0^t[c_i^*(u)-(\sigma_i^*)^2(u)/2]\rmd u\right)
\EE\left[\exp({w_i^*(t)})\right]\PP(\tau_1>t~|~\ep(0)=i).
\end{aligned}\end{equation*}
Owing to
$
\EE\left[
\exp({w_i^*(t)})
\right]=\exp\left(\frac12
\int_0^t(\sigma_i^*)^2(u)\rmd u
\right)
$
we obtain  \eqref{eq:SF*}.   

Note that by \eqref{eq:SF*} and \eqref{eq:c*g}
\begin{equation}\label{eq:gqgp}
\exp\left(
-\int_0^t\gamma_i^\QQ(u)\rmd u
\right)=\overline F_i^\QQ(t)=\exp(l_i^*(t))\overline F_i^\PP(t)=\exp\left(
\int_0^t\left(c_i^*(u)-\gamma_i^\PP(u)\right)\rmd u
\right), 
\end{equation}
\[
\quad t\geq0,\; i\in D.
\]
Hence,
$\gamma_i^\QQ\equiv\gamma_i^\PP-c_i^*$.
Since by \eqref{eq:DoobMeyer*}
$c_i^*=-h_i^*\gamma_i^\PP,$ and 
 \eqref{eq:HRF*} is completely proved.
The non-exploding condition \eqref{eq:blowupQ}
follows from \eqref{eq:c*g}.

The part (b) of the theorem follows from the  classical Girsanov's Theorem, see e.g. 
\cite{JYC}, Proposition 1.7.3.1. 

The part (c)  holds by the following observation. 
The Wiener part $W^\sigma$ of   
process
$X$ under measure $\QQ$ becomes $W^\sigma(t)=\widetilde W^\sigma(t)+T^{\sigma\sigma^*}(t)$,
see part (b).
Here $\widetilde W^\sigma$ is the $\QQ$-Wiener process, 
i.e. the It\^o integral w.r.t. $\QQ$-Brownian motion $\widetilde B$,
and $T^{\sigma\sigma^*}$ is
the $\QQ$-telegraph process which is driven 
by the subsequently switching velocities $\sigma_i(t)\sigma_i^*(t),\; i\in D$.

Therefore, under measure $\QQ$ the process $X$ is still the jump-diffusion  process with switching  regimes,
\[
X(t)=T^{c+\sigma\sigma^*}(t)+N^h(t)+\widetilde W^\sigma(t),\qquad t>0,
\]
characterised by $\langle c+\sigma\sigma^*, h, \sigma\rangle$. The theorem is proved.
\end{proof}

\subsection{Relative entropy}

Let $\PP$ and $\QQ$ be two equivalent measures. Under the time horizon $t,\;t>0,$
the relative entropy of $\QQ$ w.r.t. $\PP$ 
is defined by the set of 
functions $H_i(t),$ $t>0,\;i\in D:$ 
\begin{equation}
\label{def:entropy}
H_i(t):=\EE_\QQ\left[\left.
\ln\RN(t)~\right|~\ep(0)=i
\right]=\EE_\PP\left[\left.
\RN(t)\ln\RN(t)~\right|~\ep(0)=i
\right],
\end{equation}
see \cite{Fritelli}.
Here  the Radon-Nikod\'ym derivative $\RN(t)=\mathcal E_t\left(X^*\right)$
 is presented by \eqref{def:RN1}-\eqref{def:RN2}.

\begin{theorem}\label{eqtheo:entropy}
Let conditions \eqref{eq:DoobMeyer*}-\eqref{eq:c*g} hold.

The relative entropy functions $H_i$ 
are expressed by 
\begin{equation}\label{eq:Hi}
H_i(t)=\EE_\QQ\left[T_i^{c^*+(\sigma^*)^2/2}(t) +N_i^{\ln(1+h^*)}(t)\right],
\qquad t\geq0,\; i\in D,
\end{equation}
and satisfy the  system of the integral equations
\begin{equation}
\label{eq:H0H1}
  H_i(t)=a_i(t)+\sum_{j\in D\setminus\{i\}}\int_0^tH_j(t-u)f_{ij}(u)\rmd u,
  \qquad t\geq 0,\;i\in D.
\end{equation}
where functions $a_i$ are defined by
\begin{equation}
\label{eq:aQ}
a_i(t)=\int_0^t%
b_i(u)\overline F_i^\QQ(u)%
\rmd u,\quad t\geq0,\; i\in D.
\end{equation}
Here
\begin{equation}
\label{eq:Fb}
\begin{aligned}
\overline F_i^\QQ(u)=&\exp\left(-\int_0^u\gamma_i^\QQ(u')\rmd u'\right),\\
b_i(u)=&\gamma_i^\PP(u)-\gamma_i^\QQ(u)
+\phi_i(u)
+\frac12\sigma_i^*(u)^2,\\
\phi_i(u)=&\left\{
\begin{aligned}
\gamma_i^\QQ(u)\ln\left[\frac{\gamma_i^\QQ(u)}{\gamma_i^\PP(u)}\right],&
\quad\text{~~if~~}\gamma_i^\PP(u)\neq0,\\
0,&\quad\text{~~if~~}\gamma_i^\PP(u)=0,
\end{aligned}
\right.\\
&u\geq0,\; i\in D.
\end{aligned}
\end{equation}
\end{theorem}

\begin{proof}
Owing to \eqref{def:RN1}
\begin{equation}
\label{eq:H}
H_i(t)=\EE_\QQ\left[\left.
\ln\RN(t)~\right|~\ep(0)=i
\right]=\EE_\QQ\left[
T_i^{c^*-(\sigma^*)^2/2}(t)+N_i^{\ln(1+h^*)}(t)+W_i^{\sigma^*}(t)
\right],
\end{equation}
where the alternating tendencies $T_i^{c^*-(\sigma^*)^2/2}$, the jump process $N_i^{\ln(1+h*)}$ and the Wiener
process $W_i^{\sigma^*}$
are defined by \eqref{def:JTDsum*} (with $c_i^*-(\sigma_i^*)^2/2$ instead of $c_i^*$ and 
$\ln(1+h_i^*)$ instead of $h_i^*,\;i\in D$). 

By Theorem \ref{theo:Girsanov}, part (c),
the process
$T_i^{c^*-(\sigma^*)^2/2}(t)+N_i^{\ln(1+h^*)}(t)+W_i^{\sigma^*}(t)$
under measure $\QQ$
becomes 
$T_i^{c^*+(\sigma^*)^2/2}(t)+N_i^{\ln(1+h^*)}(t)+\widetilde W_i^{\sigma^*}(t),$
where $\widetilde W_i^{\sigma^*}(t)$ 
is the stochastic integral w.r.t. the $\QQ$-Brownian motion $\widetilde B$. 
Therefore,  $\EE_\QQ \left[\widetilde W_i^{\sigma^*}(t)\right]=0,\;t\geq0,$ and
\begin{equation*}
\begin{aligned}
H_i(t)=&
\EE_\QQ\left[
T_i^{c^*+(\sigma^*)^2/2}(t) +N_i^{\ln(1+h^*)}(t)+\widetilde W_i^{\sigma^*}(t)
\right]\\
=&\EE_\QQ\left[T_i^{c^*+(\sigma^*)^2/2}(t) +N_i^{\ln(1+h^*)}(t)\right],
\end{aligned}\end{equation*}
which gives \eqref{eq:Hi}.

Equations \eqref{eq:H0H1}-\eqref{eq:Fb} follow from Proposition
\ref{prop:mu} and Theorem \ref{theo:Girsanov}, see \eqref{eq:HRF*}.
\end{proof}

It is easy to see that  functions $b_i$  
defined by \eqref{eq:Fb}    
are non-negative, 
$b_i(u)\geq0,\;  u\geq0,$ $i\in D.$ Hence,  functions $a_i$ defined by \eqref{eq:aQ} are also non-negative,
$a_i(t)\geq0,\; t\geq0,$ $i\in D.$

\begin{remark}\label{rem:Laplace}
By applying the Laplace transform $f\to \hat f(s)=\int_0^\infty \rme^{-st}f(t)\rmd t$
to \eqref{eq:H0H1} one  can obtain the system\textup{:}
\[
\hat H_i(s)=\hat a_i(s)+\sum_{j\in D\setminus\{i\}}\hat f_{ij}(s)\hat H_j(s), \quad s>0,\qquad i\in D.
\]
if the transformations $\hat a_i(s),\;i\in D,$ exist.
The above system yields the unique solution.

For example\textup{,}  if $d=2,$ and  $b_1,\;$ $b_2$ \textup{(}see \eqref{eq:Fb}\textup{)}
 are constants\textup{;} if the alternating distributions of inter-arrival times 
are exponential\textup{,} $\gamma_i^\QQ(t)=\lambda_i^*=\const,\;i\in \{1, 2\},$ therefore
in this  simple case 
\begin{equation*}
\begin{aligned}
\hat H_1(s)=&\frac{B}{s^2}+\frac{A_1}{s+\lambda_1^*+\lambda_2^*},\\
\hat H_2(s)=&\frac{B}{s^2}+\frac{A_2}{s+\lambda_1^*+\lambda_2^*},
\end{aligned}\quad s>0,
\end{equation*}
where 
\begin{equation}\label{def:d}
A_1=\dfrac{\lambda_1^*(b_1-b_2)}{(\lambda_1^*+\lambda_2^*)^2},\quad
A_2=\frac{\lambda_2^*(b_2-b_1)}{(\lambda_1^*+\lambda_2^*)^2},\quad
B=\frac{\lambda_2^*b_1+\lambda_1^*b_2}{\lambda_1^*+\lambda_2^*}.
\end{equation}

This corresponds to the following explicit solution of \eqref{eq:H0H1}\textup{:} 
the relative entropy functions 
$H_1(t), H_2(t),\;t\geq0,$ are expressed by
\begin{equation}
\label{eq:entropy-const}
\begin{aligned}
H_1(t)=H_1(t ; \lambda_1^*, \lambda_2^*)
=&Bt%
+A_1
\left[1-\rme^{-(\lambda_1^*+\lambda_2^*) t}\right],\\ 
H_2(t)=H_2(t ; \lambda_1^*, \lambda_2^*)
=&Bt%
+A_2
\left[1-\rme^{-(\lambda_1^*+\lambda_2^*) t}\right].
\end{aligned}\end{equation}
\end{remark}

\section{Equivalent martingale measure}\label{sec:4}
\setcounter{equation}{0}

Consider  the  jump-diffusion
process
$X=T^c(t)+H^h(t)+W^\sigma(t),\;t\geq0,$ 
 with the switching  hazard rate functions $\gamma_i^\PP,\;i\in D,$  of inter-arrival times
 $T_n,\;n\geq1,$ see the definitions in \eqref{def:Tsum}-\eqref{def:Dsum}.

 Let the equivalent measure $\QQ$ be defined by 
 the Radon-Nikod\'ym density $Z(t)=\dfrac{\rmd\QQ}{\rmd\PP}|_{\mathcal F_t},\;t\geq0$, see  
 \eqref{def:RN1}-\eqref{def:RN2}.
Let driving parameters 
$c_i^*,
\;h_i^*,\;h_i^*>-1,$ and $\sigma_i^*,\;i\in D,$ 
satisfy  \eqref{eq:DoobMeyer*}-\eqref{eq:c*g}.   
By Theorem \ref{theo:Girsanov} under measure $\QQ$
 the hazard rate functions  $\gamma_i^\QQ$ are  defined by  \eqref{eq:HRF*}.

The family of the equivalent \emph{ martingale} measures for $X$ can be disclosed precisely.

\begin{theorem}\label{theo:MM}
Measure $\QQ$ is the
 martingale measure for 
process $X$
 if and only if 
\begin{equation}
\label{eq:s*}
c_i(t)+\sigma_i(t)\sigma_i^*(t)+
\gamma_i^\QQ(t)h_i(t)=0,
\quad t\geq0,\;i\in D.
\end{equation}
\end{theorem}

\begin{proof}  This result is well known, see e.g. \cite{BellamyM},  Proposition 3.1.
The proof follows from Theorem \ref{theo:DoobMeyer} and Theorem \ref{theo:Girsanov}.

Let measure $\QQ$ be defined by \eqref{def:RN1}-\eqref{def:RN2}. 
Then, by
Theorem \ref{theo:Girsanov}, part (c), under measure $\QQ$ the process 
\[
X(t)=T^{c+\sigma\sigma^*}(t)+N^h(t)+\widetilde W^\sigma(t),\qquad t>0,
\]
is again the  jump-diffusion process with switching regime.
The martingale condition \eqref{eq:DoobMeyer} of Theorem \ref{theo:DoobMeyer} 
becomes  \eqref{eq:s*}.

The theorem is proved.
\end{proof}

The relative entropy functions $H_i(t),\;t>0,\;i\in D,$ 
of the \emph{martingale} measure $\QQ$ w.r.t. $\PP$
solve system \eqref{eq:H0H1} with functions $a_i$  specified 
by \eqref{eq:aQ} and \eqref{eq:Fb}. 
Driving parameters $c_i^*, h_i^*, \sigma_i^*$ and switching intensities $\gamma_i^\PP, \gamma_i^\QQ,\;i\in D,$
satisfy  \eqref{eq:HRF*}   and  \eqref{eq:s*}.

Consider the following examples when the equivalent martingale measure $\QQ$ is unique.

\begin{example}[\textbf{Jump-telegraph process}]\label{ex1}
Consider  process $X$ missing the diffusion component\textup{,}
\[
X(t)=T^c(t)+N^h(t),\qquad t\geq0,
\]
see \eqref{def:Tsum}-\eqref{def:Jsum}.

Assume\textup{,} that $h_i(t)\neq0$ almost 
everywhere,\footnote{If\textup{,} in contrary\textup{,} $h_i(u)=0$ on 
a whole interval\textup{,} $u\in(a, b)\subset[0, \infty),$  then\textup{,}
due to \eqref{eq:s*} with $\sigma_i\equiv0,$ 
we have no martingale measures \textup{(}if   $c_i\neq0$ on the interval $(a, b)$\textup{),} or 
infinitely many ones \textup{(}if $c_i=0$ with free fragment of hazard rate function  $\gamma_i^\QQ$\textup{)}. }
and $h_i(t)$ is of the opposite sign with $c_i(t):$
\begin{equation}
\label{eq:chneg}
c_i(t)/h_i(t)<0,\quad t>0,\qquad  i\in D.
\end{equation}
Moreover\textup{,} let functions ${c_i}/{h_i},\;i\in D,$ be locally integrable and  
\begin{equation}\label{*}
\int_0^\infty\frac{c_i(u)}{h_i(u)}\rmd u=-\infty,\qquad i\in D.
\end{equation}
Then\textup{,}  by Theorem \textup{\ref{theo:MM}}
 the equivalent  martingale measure $\QQ$ exists  and it is unique
 with 
 the hazard rate functions of interarrival times defined by
 \begin{equation}
\label{eq:aQunique}
\gamma_i^\QQ(t)=-c_i(t)/h_i(t)>0,\qquad t>0,\quad i\in D.
\end{equation}
Here \eqref{*} is the non-exploding condition for measure $\QQ$.
The corresponding 
measure transformation is determined by the  functions $c_i^*, h_i^*,\;i\in D,$
\[
c_i^*(t)=\gamma_i^\PP(t)-\gamma_i^\QQ(t),
\qquad
h_i^*(t)=-1+\gamma_i^\QQ(t)/\gamma_i^\PP(t),\quad t>0,\quad i\in D,
\]
if $\gamma_i^\PP(t)>0\; a.e.,$ see \eqref{eq:HRF*} and \eqref{eq:aQunique}.
\footnote{Measures $\QQ$ and $\PP$ are equivalent.
If  $\PP$-distribution of the interarrival times has a  {``}dead{''} zone\textup{:} 
$\gamma_i^\PP(u)\equiv0,$  $u\in(a, b),$  for some time interval $(a, b)\subset[0, \infty),$ 
then due to  \eqref{eq:HRF*} and \eqref{eq:s*}
for any martingale measure $\QQ$ the hazard rate function $\gamma_i^\QQ$ also vanishes on $(a, b),$
$\gamma_i^\QQ(u)\equiv0$ and $c_i(u)=c_i^*(u)\equiv0,\; u\in(a, b).$}

 The entropy functions $H_i(t),\;i\in D,$ solve  system \eqref{eq:H0H1} with 
\[
a_i(t)=\int_0^t\left[
\gamma_i^\PP(u)-\gamma_i^\QQ(u)
+\gamma_i^\QQ(u)\ln\frac{\gamma_i^\QQ(u)}{\gamma_i^\PP(u)}
\right]\overline F_i^\QQ(u)
\rmd u,
\]
where $\gamma_i^\QQ$ are defined by \eqref{eq:aQunique}. The survival functions $\overline F_i^\QQ,\; i\in D,$
are defined in
\eqref{barFf}.

 If the inequalities \eqref{eq:chneg} do not hold\textup{,} the martingale measure does not exist.

In particular,
if the    parameters $c_i, h_i,\; h_i\neq0,$ are constant such that $c_i/h_i<0,\;i\in D,$ 
with exponentially distributed inter-switching times, $\gamma_i^\PP=\lambda_i,\;i\in D,$  
then by \eqref{eq:aQunique} under the martingale measure $\QQ$ the inter-switching times are 
again exponentially distributed with the   switching 
intensities $\lambda_i^\QQ=-c_i/h_i,\;i\in D.$  
If $d=2$, then the closed form of the entropy functions is found, see     
Remark \textup{\ref{rem:Laplace},}   formulae \eqref{def:d}-\eqref{eq:entropy-const}\textup{,}
and more detailed analysis in
Section \textup{\ref{sec:const}} below.
In this case the unique martingale measure $\QQ$ is defined by the 
Radon-Nikod\'ym density \eqref{def:RN1}-\eqref{def:RN2}  with constant $c_i^*$ and $h_i^*$\textup{:}
\begin{equation}
\label{eq:c*h*classic}
c_i^*=\lambda_i-\lambda_i^*,\qquad h_i^*=-1+\frac{\lambda_i^*}{\lambda_i},
\end{equation}
where $\lambda_i^*=-c_i/h_i>0,\;i\in D=\{1, 2\},$ are the new
 alternating intensities of the inter-switching times\textup{,}
see   e.g. \textup{\cite{KR}.}

Formulae  \eqref{def:d}-\eqref{eq:entropy-const} for the entropy functions $H_i(t),\; t\geq0,$ hold with
\[
b_i=\lambda_i+\frac{c_i}{h_i}-\frac{c_i}{h_i}\ln\left[-\frac{c_i}{h_i\lambda_i}\right],\qquad i\in D=\{1, 2\}.
\]
\end{example}

\begin{example}[\textbf{Diffusion process}]\label{ex2}
  Consider the diffusion process missing the jump component and switching\textup{,}
   $c$ is locally integrable and functions $\sigma,\;$ $c/\sigma$ are locally square integrable.
  Assume that $\sigma(u)\neq0~a.e.$
 
  Let $\QQ$ be an equivalent measure.
  In this case the Radon-Nikod\'ym derivative of $\QQ$ is defined by 
  \[
  \RN(t)=\exp\left(
  \int_0^t\sigma^*(u)\rmd B(u)-\frac12\int_0^t\sigma^*(u)^2\rmd u
  \right),
  \]
  where $\sigma^*$ is the  locally square integrable function. 
  By Girsanov\textup{'}s Theorem the process
  \[
  \widetilde B=B-\int_0^t\sigma^*(u)\rmd u
  \]
  is $\QQ$-Brownian motion.   Hence\textup{,}    process $X$ takes the form
  \[X(t)=\int_0^t\left[c(u)+\sigma(u)\sigma^*(u)\right]\rmd u
  +\int_0^t\sigma^*(u)\rmd\widetilde B(u),\qquad t\geq0.\]  This is a martingale if and only if 
  $\sigma\sigma^*\equiv-c$.
  
By \eqref{def:entropy} the relative entropy $H(t)$ of $\QQ$ w.r.t. $\PP$
  is 
  \[\begin{aligned}
  H(t)=\EE_\QQ\left[
  \ln\RN(t)
  \right]
  =&\EE_\QQ\left[
  \int_0^t\sigma^*(u)\rmd B(u)-  \frac12\int_0^t\sigma^*(u)^2\rmd u
  \right]\\
=& \EE_\QQ\left[
\int_0^t\sigma^*(u)\rmd \widetilde B(u)+  \frac12\int_0^t\sigma^*(u)^2\rmd u
  \right]\\
  =&\frac12\int_0^t\sigma^*(u)^2\rmd u.
  \end{aligned}\]

  Therefore, the relative entropy of the \textup{(}unique\textup{)} martingale measure is given by
  \[
 H(t)=\frac12\int_0^t\left[\frac{c(u)}{\sigma(u)}\right]^2\rmd u,\qquad t>0.
  \]
\end{example}

\begin{remark}[\textbf{Diffusion process with switching tendencies and diffusion coefficients}]\label{rem:TDP}
Consider  the case of the \emph{diffusion process,} 
\[
X(t)=T^c(t)+W^\sigma(t),\qquad t\geq0,
\]
with the switching tendencies $c_i=c_i(t)$ 
and diffusion coefficients $\sigma_i=\sigma_i(t)\neq0,\;t>0,\;i\in D,$
where the jump component is missing.

In this case there are infinitely many equivalent martingale measures. 

Theorem \textup{\ref{theo:MM}} shows   
that the measure transformation defined by $\RN\mid_{\mathcal F_t}=\mathcal E_t(X^*),\;$ see \eqref{def:RN1} with
$h_i^*=0,$ $c_i^*=0$ and 
 $\sigma_i^*(t)=-c_i(t)/\sigma_i(t),\; t\geq0,\;i\in D,$ 
 eliminates the drift component similarly as in Example \textup{\ref{ex2}.}
 Under measure $\QQ$
 process $X$ becomes the martingale of the form   
$
 X(t)\equiv \widetilde W^\sigma(t),\; t\geq0,
$
 whereas by \eqref{eq:s*}
 the inter-switching times are \emph{arbitrary distributed}. Here $\widetilde W^\sigma$ is the 
 stochastic integral \eqref{def:Dsum} based on 
 $\QQ$-Brownian motion $\widetilde B.$
 
This model has been analysed in \textup{\cite{ElliottAnnFin}} by using  the  Esscher transform under switching regimes.
  This transformation does not affect the distribution of inter-switching times and the corresponding
  equivalent martingale measure 
  is of  the minimal relative entropy, see  \textup{\cite{ElliottAnnFin},} Proposition \textup{3.1.}

In the next section we study in detail the jump-diffusion model with switching  regimes based on the 
Markov underlying process $\ep$. We have discovered that in this case
the Esscher transform does not produce the minimal relative entropy.
  \end{remark}

\section{
Esscher transform and minimal entropy martingale measure }\label{sec:const}
\setcounter{equation}{0}

Typically, the jump-diffusion model with switching regime  has no martingale measure or it has infinitely many.
The rare examples of the unique martingale measure are presented above 
(Example \ref{ex1} and Example \ref{ex2}). 
 In this section we discuss the case when the infinitely many martingale measures exist
 and discuss some methods to select one.
The first method is based on the so-called  the Esscher transform under switching regimes.

Let $X=T^c(t)+N^h(t)+W^\sigma(t),\;t\geq0,$
be the jump-diffusion process with switching regime, 
see \eqref{def:Tsum}-\eqref{def:Dsum},
defined on the filtered probability space $(\Omega, \mathcal F, \mathcal F_t, \PP)$.

 Let  
 $\sigma_i\neq0,\; i\in D,$ a. s.   
The case with missed diffusion ($\sigma_i\equiv0,\;i\in D$) is analysed in Example \ref{ex1}.

 To choose an equivalent martingale measure by a reasonable way
  consider   the deterministic measurable functions  $\theta_i=\theta_{i}(t),\;t\geq0,\;i\in D,$
which define the regime switching processes $\theta_i^\dagger=\theta_i^\dagger(t),\;$  $t\geq0,\;i\in D,$ 
similarly as in \eqref{def:transform}.
Let measure $\QQ_\theta$ (equivalent to $\PP$) be defined 
by the density
\begin{equation}
\label{def:Esscher}
\frac{\rmd \QQ_\theta}{\rmd\PP}\mid_{\mathcal F_t}
:=\frac{\exp\left(\int_0^t\theta^\dagger(s)\rmd Y(s)\right)}
{\EE_\PP\left[
\exp\left(\int_0^t\theta^\dagger(s)\rmd Y(s)\right)~\left|~\mathcal F_t^\ep\right.
\right]}.
\end{equation}
Here  $Y(t)=T^{c-\sigma^2/2}(t)+N^{\ln(1+h)}(t)+W^\sigma(t),\;t\geq0,$ see   \eqref{eq:stochexp}, and 
$\mathcal F_t^\ep$ is the $\PP$-augmentation of the natural filtration generated by $\ep$.
This particular choice of the new measure  is named a regime switching Esscher transform (or exponential tilting),
see Elliott et al \cite{ElliottAnnFin}.

It is easy to see that
\[
{\EE_\PP\left[
\exp\left(\int_0^t\theta^\dagger(s)\rmd Y(s)\right)~\left|~\mathcal F_t^\ep\right.
\right]}
=\exp\left(
T^{\theta(c-\sigma^2/2)}(t)+N^{\theta\ln(1+h)}(t)\right)
\exp\left(T^{\theta^2\sigma^2/2}(t)
\right).
\]
Therefore the Radon-Nikod\'ym derivative of the Esscher transforms is given by
\[
\frac{\rmd \QQ_\theta}{\rmd\PP}\mid_{\mathcal F_t}=\exp\left(
W^{\theta\sigma}(t)-T^{\theta^2\sigma^2/2}(t)
\right),
\]
which corresponds to Radon-Nikod\'ym derivative \eqref{def:RN1} with 
\begin{equation}\label{eq:c*0h*0}
\sigma_i^*=\theta_i\sigma_i,\quad
c_i^*=0,\quad h_i^*=0,\qquad i\in D.
\end{equation}
 Observe, that by Girsanov's Theorem (see Theorem \ref{theo:Girsanov}, equation 
 \eqref{eq:HRF*})
 due to \eqref{eq:c*0h*0} the distribution of inter-switching times  are not changed under such defined measure, 
 $\gamma^{\QQ_\theta}\equiv\gamma^\PP$.
 
Hence, due to \eqref{eq:c*0h*0},
 the martingale condition (see Theorem \ref{theo:MM}, equation \eqref{eq:s*}) 
can be written as 
  \begin{equation}\label{eq:theta}
\theta_i(t)=-\frac{c_i(t)+\gamma_i^{\PP}(t)h_i(t)}{\sigma_i(t)^2},\quad t>0,\; i\in D.
\end{equation}
 It is known that the Esscher measure transform defined by \eqref{def:Esscher} with parameters 
 $\theta_i,\;i\in D,$ determined by \eqref{eq:theta}
 corresponds to the \emph{minimal relative entropy}, see
\cite{ElliottAnnFin}, Proposition 3.1.
The similar approach with the Esscher measure transform produces the  { minimal relative entropy}
in the case of L\'evy processes (see \cite{EscheSchweizer,FM}).

For our model based on  Brownian motion \emph{with jumps and with switching regimes }
the Esscher transform  does not produce the minimal relative entropy.

In the rest of  this section for the sake of simplicity,
 we consider the Markov case with $d=2$, when 
the alternating distributions of inter-switching times 
are exponential both under measure $\PP$ and under an equivalent measure $\QQ$\textup{,} 
i.e. \[\gamma_i^\PP=\lambda_i=\const>0, \qquad \gamma_i^\QQ=\lambda_i^*=\const>0,\quad i\in D=\{1, 2\},\]
and the driving parameters $c_i, h_i, \sigma_i,\;i\in D=\{1, 2\},$ are constant.
Here measure $\QQ$ is defined by \eqref{def:JTDsum*}-\eqref{def:RN2}
with constant parameters $c_i^*, h_i^*, \sigma_i^*,\; i\in D=\{1, 2\},$
satisfying \eqref{eq:DoobMeyer*}-\eqref{eq:c*g}.

To analyse  the set of equivalent martingale measures from the viewpoint of the relative entropy 
we are looking for the solution of 
the integral equations  \eqref{eq:H0H1}. 
Since this   jump-diffusion process  $X$ is bounded, 
 by Theorem 2.1 of \cite{Fritelli}   there exists 
 a unique minimal entropy martingale measure.   

As it is shown in Remark \ref{rem:Laplace}, 
in this case  the relative entropy functions are defined by 
\eqref{eq:entropy-const}:
\begin{equation*}
\begin{aligned}
H_1(t)=H_1(t ; \lambda_1^*, \lambda_2^*)
=&Bt
+A_1
\left[1-\rme^{-(\lambda_1^*+\lambda_2^*) t}\right],\\ 
H_2(t)=H_2(t ; \lambda_1^*, \lambda_2^*)
=&Bt
+A_2
\left[1-\rme^{-(\lambda_1^*+\lambda_2^*) t}\right],
\end{aligned}
\qquad t\geq0,
\end{equation*}
where $A_1, A_2$ and $B$ are defined by \eqref{def:d} and \eqref{eq:Fb}.

\begin{remark}
Note that $B=0$ \textup{(}or\textup{,} equivalently\textup{,} $b_1=b_2=0)$
 if and only if process $X$ is already 
a $\PP$-martingale. Indeed, $b_1=b_2=0$ if and only if
  $\sigma_i^*=0,\;i\in D=\{1, 2\},$
 and $\lambda_i^*=\lambda_i,\;i\in D=\{1, 2\},$ 
  \textup{(}see  \eqref{eq:Fb}\textup{)}
 \textup{,} or equivalently\textup{,} 
 $\QQ=\PP$. Hence $A_1=A_2=0$ and  $H_1(t)=H_2(t)\equiv0$.
\end{remark}

\begin{remark}\label{rem:5.2}
Let the jump-diffusion process $X$ be a L\'evy process\textup{,}
 i.e. the alternation is missing
  and $c_1=c_2=c,\; h_1=h_2=h\neq0,\; \sigma_1=\sigma_2=\sigma\neq0$ are constant.
 This is the case of a \textup{ Markov jump-diffusion process.}
 
 Let the new measure $\QQ$ be defined by \eqref{def:JTDsum*}-\eqref{def:RN2}.

 Therefore\textup{,}  by \eqref{eq:entropy-const}
 the relative entropy functions are identical and linear in $t,$
\[ H_1\equiv H_2=Bt,\;t\geq0,\]
 where\textup{,} due to 
\eqref{def:d}\textup{,} $B=b=\lambda-\lambda^*+\lambda^*\ln\lambda^*/\lambda
+(\sigma^*)^2/2,$
 and $A_1=A_2=0.$
Here $\lambda=\gamma^\PP$ and $\lambda^*=\gamma^\QQ$ are the constant jump intensities
under measure $\PP$ and measure $\QQ$ respectively\textup{;}
the parameter $\sigma^*$ satisfies martingale condition \eqref{eq:s*}\textup{:}
\[
\sigma^*=-\frac{c+\lambda^* h}{\sigma}.
\]

In this case the martingale measure with the minimal relative entropy is defined
by the jump intensity $\lambda^*$ which satisfies the algebraic equation\textup{:}
\begin{equation}\label{cond:ME}
b'(\lambda^*)\equiv\ln\lambda^*/\lambda+\frac{h^2}{\sigma^2}\left(
\frac{c}{h}+\lambda^*
\right)=0.
\end{equation}

The latter equation is equivalent to
\begin{equation}\label{C:FM}
c+\beta^*\sigma^2+\lambda h\exp(\beta^* h)=0,
\end{equation}
where the following change of variables $\lambda^*=\lambda\exp(\beta^*h)$ is applied.
In this particular case of L\'evy process $X,$
equation \eqref{C:FM} coincides with condition \textup{(C)} of \textup{\cite{FM},} which gives the 
minimal relative entropy $H(t)$ under a measure defined by the Esscher transformation\textup{.}
In this example equation \eqref{eq:entropy-const} becomes
\begin{equation}\label{E:FM}
H(t)=\left[
\lambda(1-\exp(\beta^*h)+\beta^*h\exp(\beta^*h))+\frac{h^2}{2\sigma^2}\left(\frac{c}{h}+\lambda\exp(\beta^*h)\right)^2
\right]t,
\end{equation}
where $\beta^*$ is the \textup{(}unique\textup{)} solution of \eqref{C:FM}.
Equation \eqref{E:FM} corresponds to equation \textup{(3.9)} from \textup{\cite{FM}.}

In the case of the jump-diffusion process with \textup{alternating parameters,} such coincidence is not available\textup{.}
In this case the minimal entropy functions depend on the initial state and they have a bit more complicated behaviour. 
\end{remark}

Observe that 
functions $b_1=b_1(\lambda_1^*, \sigma_1^*),\; b_2=b_2(\lambda_2^*, \sigma_2^*)$ 
are expressed by summing up of the  two nonnegative and convex functions,
$f(x)=a-x+x\ln(x/a),\; x>0,$ (with $\lambda^*$ for $x$) and $g(y)=y^2/2$ (with $\sigma^*$ for $y$).
Hence, $b_1$ and $b_2$ are  nonnegative and convex. Therefore, function
$B=B(\lambda_1^*, \lambda_2^*, \sigma_1^*, \sigma_2^*)
=\dfrac{\lambda_2^*b_1+\lambda_1^*b_2}{\lambda_1^*+\lambda_2^*}$ 
is also nonnegative.

We analyse the relative  entropy functions  
$H_1=H_1(t)$ and $H_2=H_2(t)$ for small and big times $t$ separately.
These functions  possesses the following time-asymptotics.

\begin{proposition}\label{proposition}
Let the relative entropy functions $H_1$ and $H_2$ be  defined by \eqref{eq:entropy-const}. Thus\textup{,} 
\begin{equation}
\label{sim:short}
H_1(t)\sim b_1t,\qquad H_2(t)\sim b_2t,\qquad t\to0,
\end{equation}
and
\begin{equation}
\label{sim:long}
H_1(t)\sim Bt+A_1,\qquad
H_2(t)\sim Bt+A_2,
\qquad t\to\infty.\end{equation}
\end{proposition}
\begin{proof}
As $t\to0$ formulae \eqref{def:d}-\eqref{eq:entropy-const} lead to \eqref{sim:short}:
\[
H_i(t)=Bt+A_i\left[1-\rme^{-(\lambda_1^*+\lambda_2^*) t}\right]\sim\left[
B+A_i(\lambda_1^*+\lambda_2^*)
\right]t\equiv b_it,\qquad i\in D=\{1, 2\}.
\]
Long-term asymptotic \eqref{sim:long} is evident.
\end{proof}

 We choose the equivalent  measure minimising the relative entropy function
$\mathbf H(t)$ under the martingale condition \eqref{eq:s*}, i.e. $\lambda_i^*,\; \sigma_i^*,\; i\in D=\{1, 2\},$
satisfy the following relations:
\begin{equation}
\label{eq:cond}
\left\{
\begin{aligned}
c_1+\lambda_1^*\cdot h_1+\sigma_1^*\cdot\sigma_1=0,\\
c_2+\lambda_2^*\cdot h_2+\sigma_2^*\cdot\sigma_2=0.
\end{aligned}
\right.
\end{equation}

If   measure $\QQ$ is constructed by way of  minimising $b_1$ and $b_2$,  
we say that $\QQ$ is
the short-term minimal entropy martingale measure 
(MEMM), see \eqref{sim:short};   
 measure $\QQ$ is the long-term MEMM,  if $\QQ$ minimises 
 $B$, see \eqref{sim:long}.

The short-term and the long-term minimal entropy equivalent martingale measures 
are defined as the solutions of the following optimisation problems subject to martingale condition \eqref{eq:cond}\textup{:}
\begin{itemize}
\item
the short-term  MEMM is defined by 
solving 
the problem
w.r.t.
$\sigma_i^*$ and $\lambda_i^*,\;i\in D=\{1, 2\},$
\begin{equation}
\label{eq:short}
\left\{
\begin{aligned}
b_1=\lambda_1-\lambda_1^*+\lambda_1^*\ln\lambda_1^*/\lambda_1
+(\sigma_1^*)^2/2
&\to\min,\\ 
b_2=\lambda_2-\lambda_2^*+\lambda_2^*\ln\lambda_2^*/\lambda_2
+(\sigma_2^*)^2/2
&\to\min.
\end{aligned}
\right.
\end{equation}

\item
the long-term MEMM is defined by
solving the problem w.r.t.
$\sigma_i^*$ and $\lambda_i^*,\;i\in D=\{1, 2\},$
\begin{equation}
\label{eq:long}
B=\frac{\lambda_2^*b_1+\lambda_1^*b_2}{\lambda_1^*+\lambda_2^*}\to\min.
\end{equation}
\end{itemize}

\begin{theorem}\label{theo:5.1}
The solutions of  problems
\eqref{eq:short} and \eqref{eq:long} subject to condition 
\eqref{eq:cond} exist and they are unique\textup{.} 
\end{theorem}

\begin{proof}
If $\sigma_i\neq0$, then \eqref{eq:cond} gives 
\begin{equation}\label{eq:sigma*}
\sigma_i^*=-\frac{c_i+\lambda_i^*h_i}{\sigma_i}.
\end{equation}
Hence problem \eqref{eq:short} is equivalent to minimisation of the function
 \[b_i=b_i(\lambda_i^*):=\lambda_i-\lambda_i^*
+\lambda_i^*\ln(\lambda_i^*/\lambda_i)
+\frac{1}{2\sigma_i^2}\left(c_i+\lambda_i^*h_i\right)^2
,\quad i\in D=\{1, 2\}.\]
If, additionally, $h_i=0$, then  problem \eqref{eq:short} has the unique solution $\lambda_i^*=\lambda_i$. 
In this case we return to the Markov modulated  diffusion process, 
see Remark \ref{rem:TDP} and \cite{ElliottAnnFin}. 
This confirms again that in this case the minimum of entropy and the Esscher transform \eqref{def:Esscher}
lead to the same result.

If $\sigma_i=0,\;h_i\neq0$, then we have jump-telegraph model (see Example \ref{ex1}).
In this case
the martingale condition \eqref{eq:cond} gives $\lambda_i^*=-c_i/h_i$,
which corresponds to the unique equivalent martingale measure.

On the contrary, if $\sigma_i\neq0,\;h_i\neq0$, then the minimal entropy and the Esscher transform 
\eqref{def:Esscher} give different results.
Observe that 
$b_i,\;i\in D=\{1, 2\},$ 
with $\sigma_i^*$ given by \eqref{eq:sigma*}
can be rewritten as
 \begin{equation}\label{eq:b0b1A}
 b_i=b_i(\lambda_i^*):=\lambda_i-\lambda_i^*
+\lambda_i^*\ln(\lambda_i^*/\lambda_i)
+\frac{C_i^2}{2}\left(\lambda_i^*-\alpha_i\right)^2
,\quad i\in D=\{1, 2\},
\end{equation}
where $C_i^2=h_i^2/\sigma_i^2>0$ and $\alpha_i=-c_i/h_i>0,\;i\in D=\{1, 2\}.$

Differentiating \eqref{eq:b0b1A} we have:
\begin{equation}\label{eq:der}
\begin{aligned}
b_1'(\lambda_1^*)\equiv &
\frac{h_1^2}{\sigma_1^2}(c_1/h_1+\lambda_1^*)+\ln(\lambda_1^*/\lambda_1),\\
b_2'(\lambda_2^*)\equiv &
\frac{h_2^2}{\sigma_2^2}(c_2/h_2+\lambda_2^*)+\ln(\lambda_2^*/\lambda_2).
\end{aligned}\end{equation}
We remark that
functions $b_1'$ and $b_2'$ vary monotonically from $-\infty$ to $+\infty$ as 
$\lambda_1^*$ and $\lambda_2^*$ increase from 0 to $+\infty$. Hence, the system
$b_1'(\lambda_1^*)=0,\; b_2'(\lambda_2^*)=0$ (and minimisation problem  \eqref{eq:short}) has the unique solution.

Moreover, 
 if  $\lambda_i^*<-c_i/h_i,$ then the solution $\lambda_i^*$
 of the corresponding equation $b_i'(\lambda_i^*)=0$ satisfies
 $\lambda_i^*,\;\lambda_i^*>\lambda_i,\;i\in D=\{1, 2\};$
if $\lambda_i^*>-c_i/h_i,$  then $\lambda_i^*<\lambda_i,\;i\in D=\{1, 2\}.$
Therefore\textup{,} the solution  $\lambda_i^*$  of \eqref{eq:short} 
is always between 
$\lambda_i$ and $-c_i/h_i,$ $i\in D=\{1, 2\},$
 whereas the Esscher transform gives $\lambda_i^*=\lambda_i,$ $i\in D=\{1, 2\}$.

We solve  the problem \eqref{eq:long} by differentiating again. System 
\[
\frac{\pd B}{\pd\lambda_1^*}=0,\qquad \frac{\pd B}{\pd\lambda_2^*}=0
\]
is equivalent to
\begin{align}
\label{eq1}
\Phi_1:=(\lambda_1^*+\lambda_2^*)b_1'(\lambda_1^*)+b_2(\lambda_2^*)-b_1(\lambda_1^*)=&0,\\
\Phi_2:=(\lambda_1^*+\lambda_2^*)b_2'(\lambda_2^*)+b_1(\lambda_1^*)-b_2(\lambda_2^*)=&0.\label{eq2}
\end{align}

By a similar  reasoning as before, one can easily see that system \eqref{eq1}-\eqref{eq2} also has (unique) solution.
Indeed, owing to \eqref{eq:b0b1A} and \eqref{eq:der} we have 
\[
\Phi_1=\lambda_2^*\ln(\lambda_1^*/\lambda_1)+b_2(\lambda_2^*)
-\lambda_1
+\lambda_1^*+\frac{C_1^2}{2}(\lambda_1^*-\alpha_1)^2+C_1^2(\lambda_1^*-\alpha_1)(\lambda_2^*+\alpha_1).
\]
Thus, 
\[
\frac{\pd \Phi_1}{\pd\lambda_1^*}=1+\frac{\lambda_2^*}{\lambda_1^*}+C_1^2(\lambda_1^*+\lambda_2^*)>0.
\]
Hence, 
function $\Phi_1$ increases monotonically from $-\infty$ to $+\infty$
as $\lambda_1^*$ increases from 0 to $+\infty$, which means that equation \eqref{eq1} has the unique solution 
 $\lambda_1^*=\phi(\lambda_2^*)>0$ for any fixed positive $\lambda_2^*$. 
 Therefore system  \eqref{eq1}-\eqref{eq2} 
 is equivalent to 
 \[
 b_1'(\phi(\lambda_2^*))+b_2'(\lambda_2^*)=0.
 \]
Differentiating this identity and \eqref{eq:der}
one can see that
\[
\phi'(\lambda_2^*)=-\frac{C_2^2+1/\lambda_2^*}{C_1^2+1/\phi(\lambda_2^*)}<0.
\]
Hence, function $ b_1'(\phi(\lambda_2^*))$ decreases
as $\lambda_2^*$ goes from 0 to $+\infty$,
whereas $b_2'(\lambda_2^*)$ strictly increases from $-\infty$ to $+\infty$.   Therefore 
system \eqref{eq1}-\eqref{eq2}  has the unique solution.
\par

Then, the second derivatives of $B$ are
\[\begin{aligned}
B_{11}=&\frac{\lambda_2^*}{\lambda_1^*+\lambda_2^*}b_1''(\lambda_1^*)
-\frac{2\lambda_2^*}{(\lambda_1^*+\lambda_2^*)^3}\Phi_1(\lambda_1^*, \lambda_2^*),\\
B_{22}=&\frac{\lambda_1^*}{\lambda_1^*+\lambda_2^*}b_2''(\lambda_2^*)
-\frac{2\lambda_1^*}{(\lambda_1^*+\lambda_2^*)^3}\Phi_2(\lambda_1^*, \lambda_2^*),\\
B_{12}=B_{21}=&\frac{\lambda_1^*}{(\lambda_1^*+\lambda_2^*)^3}\Phi_1(\lambda_1^*, \lambda_2^*)
+\frac{\lambda_2^*}{(\lambda_1^*+\lambda_2^*)^3}\Phi_2(\lambda_1^*, \lambda_2^*).
\end{aligned}\]
Here 
$B_{ij}=\dfrac{\pd^2 B}{\pd\lambda_i^*\pd\lambda_j^*},\;i, j\in D=\{1, 2\}$. Hence, point 
$(\lambda_1^*, \lambda_2^*)$, fitting for system \eqref{eq1}-\eqref{eq2}, gives the minimum of $B$.
\end{proof}

\begin{figure}[hbt]
\begin{center}
\includegraphics[scale=0.85]{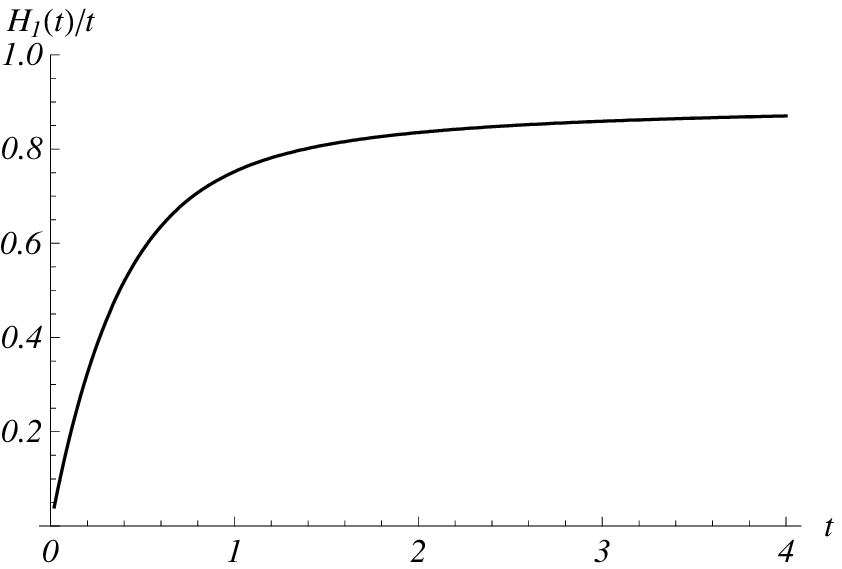}
$\,$
\includegraphics[scale=0.85]{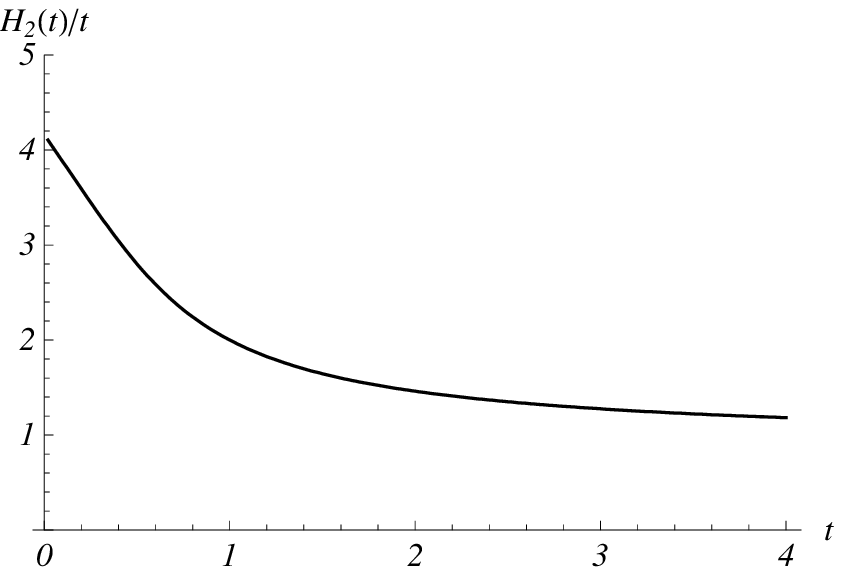}
\caption{Plot of $H_i(t)/t$, $t>0,$ where $H_i(t),\; i\in\{1, 2\},$ is the minimal entropy 
(for $\lambda_1=\lambda_2=1$, $\sigma_1=\sigma_2=1$, $c_1=-1$, $c_2=3$, $h_1=1$, $h_2=-0.1$).
}
\label{figure:1}
\end{center}
\end{figure}
%

\begin{figure}[hbt] \label{fig:2}
\begin{center}
\includegraphics[scale=0.85]{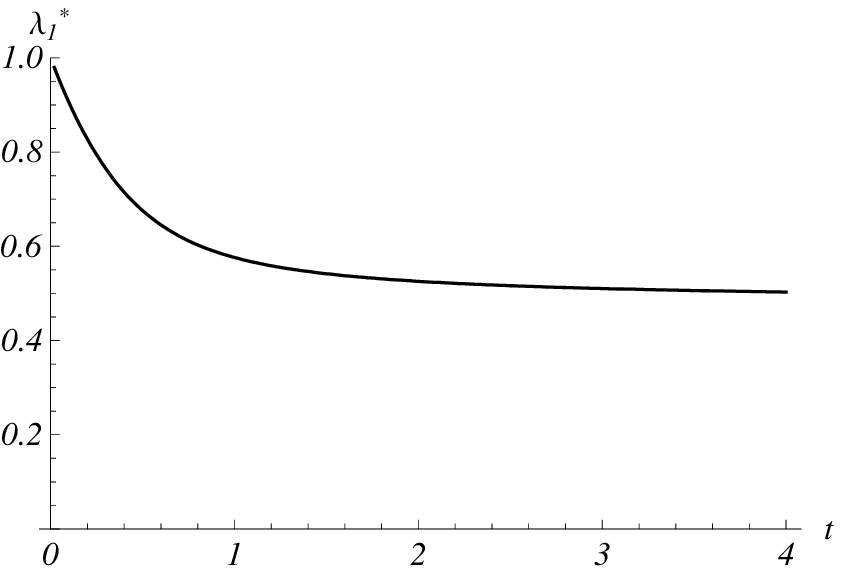}
$\,$
\includegraphics[scale=0.85]{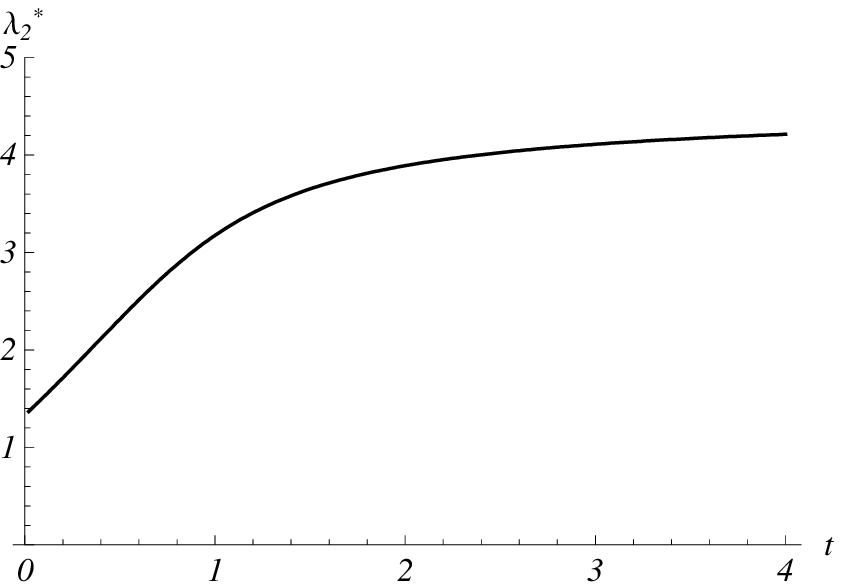}
\caption{Plot  of \textit{argmin} $H_1(t)$, for the same case of Figure \ref{figure:1}. 
}
\label{figure:2}
\end{center}
\end{figure}
%

%
\begin{remark}
Theorem \ref{theo:5.1} and Proposition \ref{proposition} show that the minimal entropy martingale 
measure differs from the measure supplied by Esscher transformation. 
This is confirmed by the plots presented in Fig.\ \ref{figure:1} and Fig.\ \ref{figure:2}, where 
an asymmetric situation is considered. 
Surprisingly, the minimal entropy is supplied by  $\lambda_1^*, \lambda_2^*$ which depend 
on time horizon. Moreover, the entropy functions $H_i(t),\;t>0,\;i\in\{1, 2\}$, are not linear 
\textup{(}cf. \eqref{E:FM}\textup{,} Remark \textup{\ref{rem:5.2}).}
 
In \textup{\cite{Elliott2007}} the Esscher transform is applied to
option pricing under a Markov-modulated jump-diffusion model.
\end{remark}

In the symmetric case, 
\begin{equation}\label{sym:case}
\lambda_1=\lambda_2,\;\sigma_1=\sigma_2~~\text{ and }~~c_1=-c_2,\; h_1=-h_2,\;c_1h_2=c_2h_1,
\end{equation}
the solution of the minimal entropy problem is constant.

\begin{proposition}
In the symmetric case \eqref{sym:case}
the problem
\begin{equation}\label{problem}
\begin{aligned}
H_1(t;\;\lambda_1^*, \lambda_2^*)&\to\min,\\
H_2(t;\;\lambda_1^*, \lambda_2^*)&\to\min
\end{aligned}\end{equation}
subject to condition \eqref{eq:cond}
has the unique solution $\lambda_1^*=\lambda_2^*=\const,$ which does not depend on time $t$.
\end{proposition}
\begin{proof}
By Taylor representation we have 
\[\begin{aligned}
H_1=&b_1t+\lambda_1^*(b_1-b_2)t^2\cdot\phi\left((\lambda_1^*+\lambda_2^*)t\right),\\
H_2=&b_2t+\lambda_2^*(b_2-b_1)t^2\cdot\phi\left((\lambda_1^*+\lambda_2^*)t\right),
\end{aligned}\]
where 
\[
\phi(x)=\sum_{n=0}^\infty\frac{(-1)^{n+1}x^n}{(n+2)!}.
\]
Notice that in the symmetric case \eqref{sym:case} function $b_1$ and $b_2$ (see \eqref{eq:b0b1A}) are equal:
$b_1(\lambda)\equiv b_2(\lambda)=b,$ $\lambda>0$. Hence, the problem \eqref{problem} 
has the unique constant solution, $\lambda_1^*=\lambda_2^*=\lambda^*$,
 which corresponds to problem \eqref{eq:short}. The minimal entropy is linear,
  $H_1(t)=H_2(t)=bt$, where $b=b(\lambda^*)$.
\end{proof}


\bibliographystyle{alea3}

\bibliography{ }

\begin{thebibliography}{99}
\bibitem[Bellami and Jeanblanc(2000)]{BellamyM}
N. Bellamy and  M.Jeanblanc.
Incompleteness of markets driven by a mixed diffusion. 
{\it Finance Stochast.} {\bf 4}, 209--222 (2000).
%
\bibitem[Cox(1962)]{renewal}
D. R. Cox.
\emph{Renewal Theory.} Wiley, New York (1962).
%
\bibitem[Di Crescenzo(2001)]{DiC01}
A. Di Crescenzo. 
On random motions with velocities alternating at Erlang-distributed random times. 
\emph{Adv. Appl. Prob.} {\bf 33}, 690--701 (2001).
%
\bibitem[Di Crescenzo and Martinucci(2010)]{DiC10}
A. Di Crescenzo  and B. Martinucci.
A damped telegraph random process with logistic stationary distribution. 
\emph{J. Appl. Prob.}  {\bf 47}, 84--96 (2010).
%
\bibitem[Di Crescenzo and Martinucci(2013)]{DiC13-1}
A. Di Crescenzo  and B. Martinucci.
On the generalized telegraph process with deterministic jumps.
\emph{Meth. Comput. Appl. Prob.} {\bf 15}, 215--235 (2013).
%
\bibitem[Di Crescenzo et al(2013)]{DiC13-2}
A. Di Crescenzo, A. Iuliano,  B. Martinucci and S. Zacks.
Generalized telegraph process with random jumps. 
\emph{J. Appl. Prob.} {\bf 50}, 450--463 (2013).
%
\bibitem[Di Crescenzo et al(2014)]{DiC14}
A. Di Crescenzo,  B. Martinucci and S.Zacks.
On the geometric Brownian motion  with alternating trend.
In:\emph{ Mathematical and Statistical Methods for Actuarial Sciences and Finance},
C.\ Perna and M.\ Sibillo, editors, pp.\ 81--85, Springer (2014).
%
\bibitem[Di Crescenzo and Zacks(2015)]{DiC13}
A. Di Crescenzo and  S. Zacks.
Probability law and flow function of Brownian motion
driven by a generalized telegraph process.
\emph{Meth. Comput. Appl. Prob. }
 {\bf 17}, 761--780 (2015)
%
\bibitem[Elliott et al(2005)]{ElliottAnnFin}
R. J. Elliott, L. Chan and T. K. Siu.
Option pricing and Esscher transform under regime switching.
\emph{Ann. Finance} {\bf 1}, 423--432 (2005).
%
\bibitem[Elliott et al(2007)]{Elliott2007}
R. J. Elliott,   T. K. Siu, L. Chan and J. W. Lau.
Option pricing under generalized Markov-modulated jump-diffusion model.
\emph{Stoch. Anal. Appl.} {\bf 25}, 821-843 (2007).
%
\bibitem[Esche and Schweizer(2005)]{EscheSchweizer}
F. Esche and M. Schweizer.
Minimal entropy preserves the L\'evy property: how and why.
\emph{Stoch. Proc. Appl.} {\bf 115}, 299--327 (2005).
%
\bibitem[F\"ollmer and Schweizer(1990)]{FS}
H. F\"ollmer  and M. Schweizer.
Hedging of contingent claims under incomplete information. 
In: \emph{Applied Stochastic Analysis}, 
M.H.A.\ Davis and R.J.\ Elliott, editors, pp.\ 101--134, 
Gordon and Breach, London (1990).
%
\bibitem[Fritelli(2000)]{Fritelli}
M. Frittelli.
The minimal entropy martingale measure and the valuation
problem in incomplete markets. 
\emph{Math. Finance} {\bf 10}, 215--225 (2000).
%
\bibitem[Fujiwara and Miyahara(2003)]{FM}
T. Fujiwara and Y. Miyahara.
The minimal entropy martingale measure for geometric L\'evy processes. 
\emph{Finance and Stochastics} {\bf 27}, 509--531 (2003).
%
\bibitem[Guo(2001)]{XG}
X. Guo.
Information and option pricings. 
\emph{Quant. Finance}, \textbf{1}, 38--44  (2001).
%

\bibitem[Jacobsen(2006)]{Jacobsen}
M. Jacobsen.
\emph{Point Process Theory
and Applications.
Marked Point and
Piecewise Deterministic
Processes}. Birkh\"auser,
Boston, Basel, Berlin (2006).
%
\bibitem[Jeanblanc and Chesney(2009)]{JYC}
M. Jeanblanc, M. Yor and M. Chesney.
\emph{Mathematical Methods for Financial Markets}. 
Springer (2009).
%
\bibitem[Jeanblanc and Rutkowski(2002)]{MM}
M. Jeanblanc and M. Rutkowski.
Default risk and hazard process. 
In: \emph{Mathematical Finance Bachelier Congress 2000}, 
Geman, H., Madan, D., Pliska, S.R., Vorst, T., editors, pp.\ 281--313. 
Springer, Berlin (2002).
%
\bibitem[Kolesnik and Ratanov(2013)]{KR}
 A. D. Kolesnik and N. Ratanov.
\emph{Telegraph Processes and Option Pricing}. 
Springer, Heidelberg (2013) 
%
\bibitem[Melnikov and Ratanov(2007)]{MR}
A. V. Melnikov and N. E. Ratanov.
Nonhomogeneous telegraph processes and their application to financial market modeling.
\emph{Doklady Math.} {\bf 75}, 115--117 (2007).
 %
\bibitem[Ratanov(2007)]{R2007}
N. Ratanov. 
A jump telegraph model for option pricing. 
\emph{Quant.  Finance} {\bf 7}, 575--583 (2007).
%
\bibitem[Ratanov(2010)]{BJPS}
N. Ratanov.
Option pricing model based on a Markov-modulated diffusion with jumps. 
\emph{Braz. J. Probab. Stat.}  \textbf{24}, 413--431 (2010).
%
\bibitem[Ratanov(2013)]{STAPRO13}
N. Ratanov.
Damped jump-telegraph processes.
\emph{Stat. Prob. Lett. } \textbf{83}, 2282--2290 (2013).
%
\bibitem[Ratanov(2014a)]{STAPRO14}
N. Ratanov.
On piecewise linear processes.
\emph{Stat. Prob. Lett.} \textbf{90}, 60--67 (2014a).
%
\bibitem[Ratanov(2014b)]{SAA}
N. Ratanov.
Double telegraph processes and complete market models.
\emph{Stoch. Anal. Appl.} {\bf  32},   555--574  (2014b).
%
%
\bibitem[Ratanov(2015)]{MCAP}
N. Ratanov.
Telegraph processes with random jumps and complete market models. 
\emph{Meth. Comput. Appl. Prob.},  {\bf 17}, 677--695 (2015).
%
\bibitem[Runggaldier(2004)]{Rung}
W. J. Runggaldier.  
Jump-diffusion models. 
In: \emph{Handbook of Heavy Tailed Distributions in Finance},  
Rachev, S.T., editor. North Holland (2004).
%
\bibitem[Weiss(1994)]{weiss}
G. H. Weiss.
\emph{Aspects and applications of the random walk}.
North-Holland, Amsterdam (1994).
%
\end{thebibliography}

\end{document}